\documentclass[reqno, 11pt]{amsart}

\makeatletter
\let\origsection=\section \def\section{\@ifstar{\origsection*}{\mysection}} 
\def\mysection{\@startsection{section}{1}\z@{.7\linespacing\@plus\linespacing}{.5\linespacing}{\normalfont\scshape\centering\S}}
\makeatother    

\usepackage{geometry}
\numberwithin{equation}{section}
\numberwithin{figure}{section}

\usepackage{xcolor}
\usepackage{hyperref}
\hypersetup{
    unicode,
    colorlinks,
    linkcolor={red!60!black},
    citecolor={green!60!black},
    urlcolor={blue!60!black}
}

\usepackage{amsfonts,amsthm,amssymb,amsmath}
\usepackage[T1]{fontenc}
\usepackage{microtype}
\usepackage{graphicx}
\usepackage{tikz}
\usepackage{tkz-berge}
\usepackage[usenames,dvipsnames]{pstricks}
\usepackage{epsfig}
\usepackage{verbatim}
\usepackage{scalerel}

\usepackage{enumerate,enumitem}
\usepackage{thmtools, thm-restate}
\usetikzlibrary{decorations.pathmorphing, arrows, calc, matrix, arrows.meta}

\usepackage{mathtools}
\usepackage{subcaption}

\newtheorem{theorem}{Theorem}

\newcommand{\N}{\mathbb{N}}

\renewcommand{\triangleleft}{\vartriangleleft}
\renewcommand{\leq}{\leqslant}

\renewcommand{\rho}{\varrho}
\renewcommand{\subset}{\subseteq}
\renewcommand{\supset}{\supseteq}
\newcommand{\nottriangleleft}{\not\kern-1pt\mathrel{\triangleleft}}

\begin{document}
\title{A unified existence theorem for normal spanning trees}   

\author{Max Pitz}

\begin{abstract}
We show that a graph $G$ has a normal spanning tree if and only if its vertex set is the union of countably many sets each separated from any subdivided infinite clique in $G$ by a finite set of vertices. 
This proves a conjecture by Brochet and Diestel from 1994, giving a common strengthening of two classical normal spanning tree criterions due to Jung and Halin.
 
Moreover, our method gives a new, algorithmic proof of Halin's theorem that every connected graph not containing a subdivision of a countable clique has a normal spanning tree.
\end{abstract}

\maketitle

\section{Overview}

A rooted spanning tree $T$ of a graph $G$ is called \emph{normal} if the end vertices of any edge of $G$ are comparable in the natural tree order of $T$. Intuitively, all the edges of $G$ run `parallel' to branches of $T$, but never `across'. Since their introduction by Jung in 1969, normal spanning trees have developed to be perhaps the single most useful structural tool in infinite graph theory. 

Every countable connected graph has a normal spanning tree, but uncountable graphs might not, as demonstrated by complete graphs on uncountably many vertices. We have the following characterisation for the existence of normal spanning tree due to Jung \cite{jung1969wurzelbaume}. Here, a set of vertices $U$ is \emph{dispersed} in $G$ if every ray in $G$ can be separated from $U$ by a finite set of vertices.

\begin{theorem}[Jung, 1969]
\label{thm_jung}
A connected graph has a normal spanning tree if and only if its vertex set is a countable union of dispersed sets.
\end{theorem}

However, Jung's condition can be hard to verify, and the most useful \emph{sufficient} condition in practice giving a normal spanning tree, see e.g.\ \cite{diestel1994depth,diestel1996classification,diestel1992proof}, is the following criterion due to Halin~\cite{halin1978simplicial}.

\begin{theorem}[Halin, 1978]
\label{thm_halin}
Every connected graph not containing a subdivision of a countable clique admits a normal spanning tree.
\end{theorem}

In 1994, Brochet and Diestel \cite[\S10, Problem~4]{brochet1994normal} proposed a common extension of Jung's and Halin's normal spanning tree criterions. The main result in this note, Theorem~\ref{thm_max}, confirms their conjecture. 
Call a set $U$ of vertices \emph{$TK^{\aleph_0}$-dispersed in $G$} if every subdivided infinite clique in $G$ can be separated from $U$ by a finite set of vertices.

\begin{theorem}
\label{thm_max}
A connected graph has a normal spanning tree if and only if its vertex set is a countable union of $TK^{\aleph_0}$-dispersed sets.
\end{theorem}

Clearly, Theorem~\ref{thm_max} implies both Theorem~\ref{thm_jung} and \ref{thm_halin}. The method of our proof of Theorem~\ref{thm_max} is also new. It consists of a single greedy algorithm, which constructs the desired normal spanning tree in just $\omega$-many steps. It differs in this respect from all the known proofs of Halin's criterion, which have relied on advanced results from structural graph theory and typically used transfinite recursions of order type $\kappa = |G|$: Halin's original proof employing his theory of simplicial decompositions \cite{halin1978simplicial}, Robertson, Seymour \& Thomas's proof using tree decompositions \cite{robertson1992excluding}, and Polat's proof using the topology of the end space \cite{polat1996ends2}. For further details on how these approaches interact, see also \cite{diestel1994depth} and \cite[Theorem~12.6.9]{Bible}.

\section{The proof}

The \emph{tree-order} $\leq_T$ of a tree $T$ with root $r$ is defined by setting $u \leq_T v$ if $u$ lies on the unique path from $r$ to $v$ in $T$. Given a vertex $v$ of $T$, we denote by $T_v := T[\{t \colon v \leq_T t\}]$ the \emph{uptree of $T$ rooted in $v$}.
For rooted trees that are not necessarily spanning, one generalises the notion of normality as follows: A rooted tree $T \subset G$ is \emph{normal (in $G$)} if the end vertices of any $T$-path in $G$ (a path in $G$ with end vertices in $T$ but all edges and inner vertices outside of $T$) are comparable in the tree order of $T$. If $T$ is spanning, this clearly reduces to the definition given in the introduction. Note that if $T \subset G$ is normal, then the set of neighbours $N(D)$ of any component $D$ of $G - T$ forms a chain in $T$, i.e.\ all vertices in $N(D)$ are comparable in $\leq_T$. 

\begin{proof}[Proof of Theorem~\ref{thm_max}]
For the forwards implication, recall that the levels of any normal spanning tree are dispersed, and hence in particular $TK^{\aleph_0}$-dispersed in $G$.

Conversely, let $G$ be a connected graph and let $\{V_n \colon n \in \N\}$ be a collection of $TK^{\aleph_0}$-dispersed sets in $G$ with $V(G) = \bigcup_{n \in \N} V_n$. 
Construct a countable chain $T_0 \subset T_1 \subset T_2 \subset \cdots$ of rayless normal trees in $G$ with the same root $r \in V(G)$ as follows: Put $T_0 = \{r\}$, and suppose $T_n$ has already been defined. Since $T_n$ is a rayless normal tree, any component $D$ of $G - T_n$ has a finite neighbourhood $N(D)$ in $T$. For each $x \in N(D)$ select one neighbour $y_x \in D$ of $x$ and call the resulting set $Y_D = \{y_x \colon x \in N(D) \} \subseteq D$. Also, let $n_D$ be minimal such that $V_{n_D} \cap D \neq \emptyset$ and pick some vertex $v_D \in V_{n_D} \cap D$. Since $G$ is connected, we may extend $T_n$ finitely into every such component $D$ as to cover $Y_D \cup \{v_D\}$ preserving normality (see \cite[Proposition~1.5.6]{Bible}), so that the extension $T_{n+1} \supset T_n$ is a rayless normal tree with root $r$. This completes the construction. 

For later use we observe that normality of $T_{n+1}$ implies that $D \cap T_{n+1}$ is connected for every component $D$ of $G - T_n$ $(\star)$. Indeed, since $T_n \subset T_{n+1}$ is a rooted subtree, it follows that any two vertices $s$ and $s'$ belonging to distinct components $S$ and $S'$ in $D \cap T_{n+1}$ will be incomparable in the tree order of $T_{n+1}$; hence any $T_{n+1}$-path in $D$ from $S$ to $S'$, which exists since $D$ is connected, violates normality of $T_{n+1}$.

Clearly, $T = \bigcup_{n \in \N} T_n$ with root $r$ is a normal tree in $G$. We show that $T$ is spanning unless some $V_n$ was not $TK^{\aleph_0}$-dispersed.

Indeed, if $T$ is not spanning, consider a component $C$ of $G-T$, and let $n_C$ be minimal such that $V_{n_C} \cap C \neq \emptyset$. Then $N(C) \subseteq T$ is a chain in $T$ which must be infinite: if it was finite, then $N(C) \subseteq T_n$ for some $n \in \N$ but then we would have extended $T_n$ into $C$, a contradiction. Hence, $N(C)$ lies on a unique ray $R \subset T$ starting at the root of $T$. 

We claim that 
\begin{enumerate}
\item every vertex $x \in N(C)$ dominates $R$, and that
\item $U := \bigcup \{ V_n \colon n \leq n_C \}$ cannot be separated from $R$ by a finite set of vertices.
\end{enumerate} 
To see this, we show that for arbitrarily high edges $e = uv \in E(R)$ with $x <_T u <_T v$, we have $N(x) \cap T_v \neq \emptyset \neq U \cap T_v$. Let $n \in \N$ be large enough such that $v \in T_n$, and let $D$ be the component of $G-T_n$ in which $R$ has a tail. Note that $C \subset D$ since $C$ has neighbours on every tail of $R$. Since $x \in N(C) \subseteq T$ we also have $x \in N(D) \subseteq T_n$, and hence by construction, there is a neighbour $y_x$ of $x$ in $G$ and a vertex $v_D \in U$ contained in $T_{n+1} \cap D$. Since the latter set is connected by $(\star)$, there exist $y_x{-}R$ and $v_D{-}R$ paths in $T_{n+1} \cap D$, which avoid $e$ and hence witness that $N(x) \cap T_v \neq \emptyset \neq U \cap T_v$.

To obtain our final contradiction, note that using (1) 
it is straightforward to construct a subdivision $K \subset G$ of a countable clique with branch vertices some subset of $N(C)$, see e.g.\ \cite[Exercise~8.30]{Bible}. By (2), the set $U$ cannot be finitely separated from $K$, i.e.\ $U$ is not $TK^{\aleph_0}$-dispersed. However, the property of being $TK^{\aleph_0}$-dispersed is preserved under finite unions; hence, one of $V_n$ for $n \leq n_C$ fails to be $TK^{\aleph_0}$-dispersed, a contradiction.
\end{proof}

A simple proof for Halin's Theorem~\ref{thm_halin} may be extracted by skipping the selection of $v_D$ and the verification of property (2). Similarly, by skipping the steps of choosing $Y_D$ and of verifying property (1), we obtain a simple proof for  Jung's Theorem~\ref{thm_jung}.

\bibliographystyle{plain}
\bibliography{reference}

\end{document}